\RequirePackage{fix-cm}
\documentclass[twocolumn]{svjour3}          % twocolumn
\smartqed  % flush right qed marks, e.g. at end of proof
\usepackage{graphicx}
\usepackage{mathptmx}      % use Times fonts if available on your TeX system
\usepackage{latexsym}
\usepackage{amsfonts}
\usepackage{amssymb}
\spnewtheorem*{theorem*}{Theorem}{\bf}{\it}
\journalname{Structural Chemistry}
\begin{document}

\title{Finite sets of operations sufficient to construct any fullerene from $C_{20}$
\thanks{This work is supported by the Russian Science Foundation under grant 14-11-00414.}
}


\author{Victor M. Buchstaber \and
        Nikolay Yu. Erokhovets %etc.
}


\institute{Victor M. Buchstaber\and 
           Nikolay Yu. Erokhovets \at
              Steklov Mathematical Institute of Russuian Academy of Sciences, Gubkina str. 8, 119991, Moscow, Russia\\
              Tel.: +7-495-9848141*3960\\
              Fax: +7-495-9848139\\
              \email{buchstab@mi.ras.ru, erochovetsn@hotmail.com}   %  \\
}

\date{Received: date / Accepted: date}

\maketitle

\begin{abstract}
We study the well-known problem of combinatorial classification of fullerenes. By a (mathematical) fullerene we mean a convex simple three dimensional polytope with all facets pentagons and hexagons. We analyse approaches of construction of arbitrary fullerene from the dodecahedron (a fullerene $C_{20}$). A growth operation is a combinatorial operation that substitutes the patch with more facets and the same boundary for the patch on the surface of a simple polytope to produce a new simple polytope. It is known that an infinite set of different growth operations transforming fullerenes into fullerenes is needed to construct any fullerene from the dodecahedron. We prove that if we allow a polytope to contain one exceptional facet, which is a quadrangle or a heptagon, then a finite set of growth operation is sufficient. We analyze pairs of objects: a finite set of operations, and a family of acceptable polytopes containing fullerenes such that any polytope of the family can be obtained from the dodecahedron by a sequence of operations from the corresponding set.  We describe explicitly three such pairs. First two pairs contain seven operations, and the last -- eleven operations. Each of these operations corresponds to a finite set of  growth operations and is a composition of edge- and two edges-truncations. 
\end{abstract}
\section{Preface}
\label{sec:pref}
For any mathematical fullerene the Euler's formula implies that the number of pentagons is twelve.  Nowadays the Euler formula and this consequence are  known far beyond the mathematics and the gratitude to Euler is expressed in the Nobel lecture of one of the laureates for discovery of fullerenes.  The natural question arises: What other restrictions on physical fullerenes follow from the combinatorics of mathematical counterparts. One of these restrictions is a mathematical result that the number of hexagons can not be equal to one. For centuries mathematics collected many characteristics to study the combinatorics of convex bodies: edge graph, Hamilton paths and cycles, belts of facets, regular colorings, perfect matchings (corresponding to Kekul\'e structures in chemistry), etc. Some of these notions found applications in the well-known problem: To obtain a fullerene with given number of hexagons as a result of reconstruction of a fullerene with less hexagons. In our article we present new results in this direction. 

Our study was inspired by the encouraging book E.A. Lord, A.L. Mackay, S. Ranganathan.  New Geometries for New Materials. Cambridge University Press,  2006 \cite{LMR06}. We submit this article to the Alan L Mackay special issue as a sign of our deep respect to this outstanding scientist and with the hope that our results will find applications in problems of physics and chemistry of fullerenes.

\section{Introduction}
\label{intro}
We will use notions and results of the polytope theory following \cite{G03,Z07}.
\begin{definition}
\label{def:full}
By a {\it (mathematical) fullerene} we mean a simple convex $3$-polytope with all facets pentagons and hexagons. {\it Simple} means that any vertex belongs to exactly three facets.  An {\it $IPR$-fullerene} is a fullerene without adjacent pentagons. 
\end{definition}
This is a mathematical model for a fullerene -- a carbon molecule synthesized in 1985 (Nobel prize 1996 in chemistry to Robert
Curl, Harold Kroto, and Richard Smalley \cite{C96,K96,S96}). Since discovery of fullerenes they were studied extensively from the mathematical point of view. As mentioned above the Euler formula implies that any fullerene contains exactly $p_5=12$ pentagons, while it is not hard to prove that the number $p_6$ of hexagons can be arbitrary nonnegative integer except for one. There is only one combinatorial type of fullerenes with $p_6=0$ (the dodecahedron $C_{20}$), while from the work by W. Thurston \cite{T98} it follows that for large values of $p_6$ the number of combinatorial types of fullerenes growth like $p_6^9$. There are effective algorhythms of combinatorial enumeration of fullerenes on supercomputers by G. Brinkmann and A. Dress \cite{BD97} and later to G. Brinkmann, J. Goedgebeur, and B.D. McKay \cite{BGM12}. For example, for $p_6=75$ there are 46088157 combinatorial types of fullerenes (see \cite{BCGM03}). Thus, the important problem is to define and study different structures and operations on the set of all combinatorial types of fullerenes. The well-known problem \cite{BGM12,BFFG06,HFM08,BF03,YF97,BGJ09,AKS16} is to find a simple set of operations sufficient to construct arbitrary fullerene from the dodecahedron $C_{20}$. 
\begin{definition}
\label{def:patch}
A {\it patch} is a disk bounded by a simple edge-cycle on the boundary of a simple $3$-polytope. Any patch is topologically a disk and the complement to a patch is also a disk.  A {\it growth operation} is a combinatorial operation that gives a new $3$-polytope $Q$ from a simple $3$-polytope $P$ by substituting  a new patch with the same boundary and more facets for the patch on the boundary of $P$.   
\end{definition}

The Endo-Kroto operation  \cite{EK92} (Fig. \ref{fig:EK}) is the simplest example of a growth operation that transforms a fullerene into a fullerene. It increases $p_6$ by one.
\begin{figure}
\includegraphics[width=0.4\textwidth]{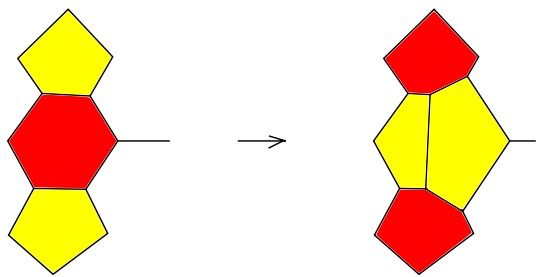}
\caption{Endo-Kroto operation}
\label{fig:EK}
\end{figure}
It was proved in \cite{BGJ09} that there is no finite set of growth operations transforming fullerenes into fullerenes sufficient to construct arbitrary fullerene from the dodecahedron. In \cite{HFM08} the example of an infinite set was found. 

We construct finite sets of growth operations and corresponding families of simple $3$-polytopes extending the set of fullerenes such that any polytope of the family can be combinatorially constructed from the dodecahedron by a sequence of these operations and all intermediate polytopes belong to the family. The less polytopes the extended family contains the more operations we need. All of the operations we construct are composition of edge- and two edges-truncations, see Definition \ref{def:sk}. 

\section{Main tools}
\label{sec:tools}

\subsection{$(s,k)$-truncations}
\label{subsec:sk-tr}

\begin{definition}
\label{def:sk}

\begin{figure}
\includegraphics[width=0.4\textwidth]{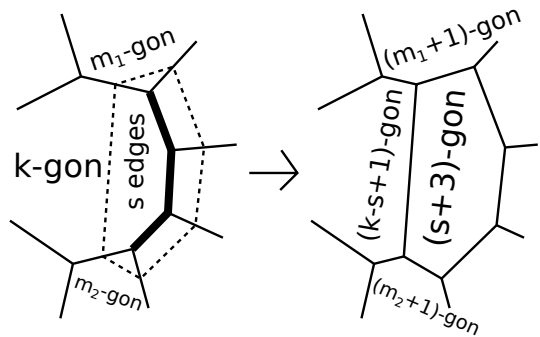}
\caption{$(s,k)$-truncation}
\label{fig:sk}
\end{figure}

Let $F$ be a {\it $k$-gonal facet} of a simple $3$-polytope. Choose {\it $s$ consequent edges} of $F$,
rotate the supporting hyperplane of $F$ around the axis passing through the midpoints of adjacent two edges (one on each side), and take the corresponding hyperplane truncation (see Fig. \ref{fig:sk}).

We call this operation an {\it $(s,k)$-truncation}. If the facets intersecting $F$ by the edges adjacent to truncated sequence of edges are $m_1$- and $m_2$-gons, then we also call the operation an {\it $(s,k;m_1,m_2)$-truncation}. The reflexion of the space $\mathbb R^3$ changes $m_1$- and $m_2$-gons, so we will not distinguish between $(s,k;m_1,m_2)$- and $(s,k;m_2,m_1)$-truncations. For $s=1$ the truncated edge belongs to two facets, so we denote these operation simply as $(1;m_1,m_2)$-truncation. 
\end{definition}

\begin{example}
\label{ex:tr}
\begin{enumerate}
\item A vertex-truncation is a $(0,k)$-truncation.
\item An edge-truncation is a $(1, k)$-truncation.
\item The Endo-Kroto operation is a $(2, 6)$-truncation.
\item $(2,6)$-truncation is the only $(s,k)$-truncation that transforms a fullerene to a fullerene.
\end{enumerate}
\end{example}

\begin{remark}
\label{rem:tr}
For $s=0$ and $s=k-2$ we have the vertex-truncation, which can be considered as the substitution of the corresponding patch for the three facets containing the vertex; hence it is a growth operation. Let $0<s<k-2$ and the $m_1$-gon and the $m_2$-gon do not intersect. Then the union of a $k$-gon and these facets form a patch, and after the $(s,k)$-truncation we have the new patch with combinatorially the same boundary (see Fig. \ref{fig:grtr} on the left). Thus we have a growth operation, which we denote by the scheme on Fig. \ref{fig:grtr} on the right.

\begin{figure}
\includegraphics[width=0.4\textwidth]{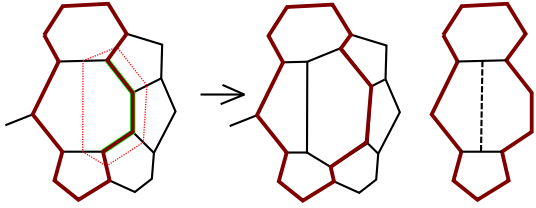}
\caption{$(s,k)$-truncation as a growth operation}
\label{fig:grtr}
\end{figure}

For $s=1$ as mentioned above the edge-truncation can be considered as a $(1,k)$-truncation and a $(1,k')$-truncation for two facets containing the truncated edge. This gives two different patches, which differ by one facet. 
\end{remark}

\begin{theorem}[see \cite{E1891,B1900}] 
\label{3ptheorem}
A  $3$-polytope is simple if and only if it is combinatorially equivalent to a polytope 
obtained from the tetrahedron by a sequence of vertex-, edge- and
$(2,k)$-truncations.
\end{theorem}
This result was used in works of famous crystallographer E.S. Fedorov.
\subsection{$k$-belts and flag polytopes}
\label{subsec:kbf}
We will denote the combinatorial equivalence of polytopes $P$ and $Q$ by $P\simeq Q$. The polytope combinatorially equivalent to the tetrahedron we call simplex and denote $\Delta^3$. 
\begin{definition}
\label{def:belt}
A cyclic sequence of $k$ facets $(F_{i_1},\dots,F_{i_k})$,  $F_{i_{k+1}}=F_{i_1}$, is called a {\it $k$-belt} if $F_{i_1}\cap\dots\cap F_{i_k}=\varnothing$, and $F_{i_p}\cap F_{i_q}\ne\varnothing$ if and only if $p-q=\pm 1\;{\rm mod}\; k$. 

A simple $3$-polytope is called {\it flag}, if it is not a simplex and has no $3$-belts. Equivalently, any set of pairwise adjacent facets has nonempty intersection.  
\end{definition}

It is easy to prove that if $P$ is a flag $3$-polytope, then the $(s,k)$-truncation transforms it to a flag  polytope if and only if $0<s<k-2$.

For the structure of $k$-belts on simple $3$-polytopes with at most hexagonal facets see \cite{E15}. Related results on cyclic $k$-edge cuts see in \cite{D98,D03,KS08,KM07,KKLS10}. 

In \cite{BE15a} the analog of Theorem \ref{3ptheorem} for flag polytopes was proved.
\begin{theorem}
\label{th:flag}
A simple $3$-polytope is flag if and only if it is combinatorially
equivalent to a polytope obtained from the cube by a sequence of edge- and $(2,k)$-truncations, $k\ge 6$.
\end{theorem}

For us this result is important, since any fullerene is a flag polytope. This follows from \cite{D98}. We will also need a more general result.
\begin{theorem}[\cite{BE15b,BE16}]
\label{th:3belts}
Let $P$ be simple $3$-polytope with $p_3=0$, $p_4\le 2$, $p_7\le
1$, and $p_k=0$, $k\ge 8$. Then it has no $3$-belts. In particular, it
is a flag polytope.
\end{theorem}
Also we will need the following results. In the case of fullerenes it follows from \cite{D03}.
\begin{theorem}[\cite{BE15b,BE16}]\label{4belts-theorem}
Let $P$ be a simple polytope with all facets pentagons and hexagons with at
most one exceptional facet $F$ being a quadrangle or a heptagon.
\begin{enumerate}
\item If $P$ has no quadrangles, then $P$ has no $4$-belts.
\item If $P$ has a quadrangle $F$, then there is exactly one $4$-belt. It
    surrounds $F$.
\end{enumerate}
\end{theorem}
Our main Theorems \ref{th:fultr}, \ref{th:trw4} and \ref{th:iptr} imply the following result.
\begin{theorem} \label{th:conf}
Any fullerene can be obtained from the dodecahedron by a sequence of $p_6$ edge- and only $(2,6)$- and $(2,7)$-truncations.
\end{theorem}
\subsection{Straightening along an edge}
\label{subsec:sag}
It is important that in some cases there is an operation inverse to the $(s,k)$-truncation.
\begin{definition}
\label{def:str}
For some edges $F_p\cap F_q$  of a simple $3$-polytope $P$ there is an operation of {\it straightening along the edge} (see Fig. \ref{fig:straight}).
\end{definition}
\begin{figure}[h]
\includegraphics[width=0.4\textwidth]{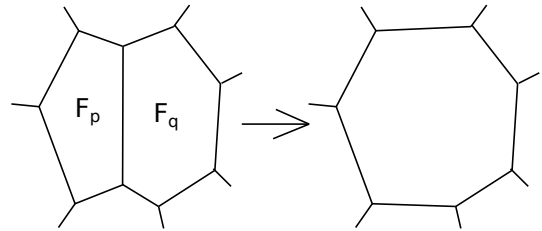}
\caption{Straightening along the edge}
\label{fig:straight}
\end{figure}
We will need the following result.
\begin{lemma}[\cite{BE15a,BE16}]
\label{lem:str}
For $P\simeq\Delta^3$ no straightening operations are defined. Let
$P\not\simeq \Delta^3$ be a simple polytope. For an edge $F_i\cap F_j$ the operation of straightening is
defined if and only if there is no $3$-belt
$(F_i,F_j,F_k)$ for some $F_k$.
\end{lemma}
\begin{remark}
\label{rem:skstr}
We see that $(s,k)$-truncation is a combinatorial operation and is always defined. It is easy to show that the straightening along the edge is a combinatorially inverse operation. Lemma \ref{lem:str} gives the criterion when it is defined.
\end{remark}
\begin{corollary} 
\label{cor:flag}
A polytope $P$ is flag if and only if straightening along any its edge is defined.
\end{corollary}
Theorem \ref{th:3belts} implies the following result.
\begin{corollary}
\label{cor:ful}
Let $P$ be a fullerene, or more generally, a simple polytope with $p_3=0$, $p_4\le 2$, $p_7\le 1$, and $p_k=0$ for $k\ge 8$. Then straightening is defined along any edge. 
\end{corollary}
In fact we will need a special case that the straightenings we use are defined. This can be also proved using the following remarkable result.
\begin{lemma}[\cite{DSS13}, Case 2 in Subsection 2.5] 
\label{lem:ndf}
Let the sphere $S^2$ be glued edge-to-edge from pentagons, hexagons, and one $k$-gon with $3\leqslant k\leqslant 7$, such that the obtained graph is $3$-valent. Then this partition is combinatorially equivalent to the boundary of a simple $3$-polytope.  
\end{lemma} 
\begin{corollary}
\label{cor:ndf}
If a straightening along the edge of a simple $3$-polytope with all facets pentagons, hexagons and one facet quadrangle or a heptagon gives a simple partition of the sphere $S^2$ with the same properties, then the straightening is well-defined.      
\end{corollary}
\section{Construction of fullerenes by truncations}
\label{sec:constr}
\subsection{Seven truncations}
\label{subsec:7tr}
\begin{definition}
\label{def:sful}
Let $\mathcal{F}_{-1}$ be the set of combinatorial simple $3$-polytopes with all facets pentagons and hexagons except for one  facet quadrangle.  The Euler formula implies that any polytope in $\mathcal{F}_{-1}$ has $p_5=10$.

Let $\mathcal{F}$ be the set of all fullerenes, and $\mathcal{F}^{IPR}$ be the set of all $IPR$-fullerenes.

Let $\mathcal{F}_1$ be the set of simple $3$-polytopes with all facets pentagons and hexagons except for one facet heptagon adjacent to a pentagon such that either there are two pentagons with the common edge intersecting the heptagon and a hexagon  (Fig. \ref{fig:7556}a), or for any two adjacent pentagons exactly one of them is adjacent to the heptagon (Fig. \ref{fig:7556}b). Any polytope in $\mathcal{F}_1$ has $p_5=13$. Let $\mathcal{F}_1^{IPR}$ be the set of polytopes in $\mathcal{F}_1$ without adjacent pentagons.  
\begin{figure}
\includegraphics[width=0.4\textwidth]{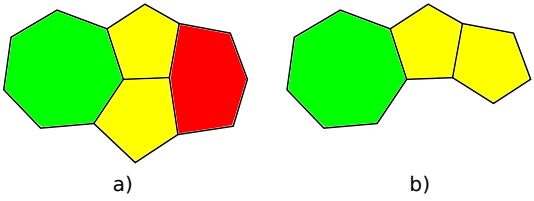}
\caption{Fragments on a polytope in $\mathcal{F}_1$}
\label{fig:7556}
\end{figure}
Set $\mathcal{F}_s=\mathcal{F}_{-1}\sqcup \mathcal{F}\sqcup\mathcal{F}_1$.
\end{definition}

\begin{theorem}[\cite{BE15b,BE16}] 
\label{th:fultr}
Any polytope in $\mathcal{F}_s$ can be obtained from the dodecahedron by a sequence of $p_6+2p_7-p_4$ truncations: $(1;4,5)$-, $(1;5,5)$-, $(2,6;4,5)$-, $(2,6;5,5)$-,\linebreak  $(2,6;5,6)$-,  $(2,7;5,5)$-, and $(2,7;5,6)$-, in such a way that intermediate polytopes belong to $\mathcal{F}_s$. Moreover (see Fig. \ref{fig:trgraph}),
\begin{enumerate} 
\item any polytope in $\mathcal{F}_{-1}$ can be obtained by a $(1;5,5)$- or $(1;4,5)$-truncation from a polytope in  $\mathcal{F}$ or $\mathcal{F}_{-1}$ respectively;
\item any polytope in $\mathcal{F}_1$ can be obtained by a $(2,6;5,6)$- or  $(2,7;5,6)$-truncation from a polytope in $\mathcal{F}$ or  $\mathcal{F}_1$ respectively; 
\item polytope in $\mathcal{F}$ can be obtained by a  $(2,6;5,5)$-, \linebreak $(2,6;4,5)$-, or $(2,7;5,5)$-truncation from a polytope in $\mathcal{F}$, $\mathcal{F}_{-1}$ or $\mathcal{F}_1$ respectively.
\end{enumerate}
\end{theorem}

\begin{figure}
\includegraphics[width=0.4\textwidth]{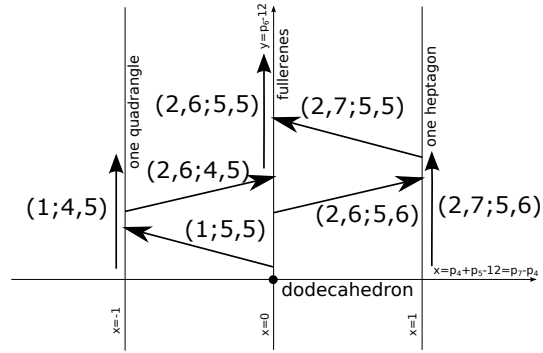}
\caption{Scheme of the truncation operations}
\label{fig:trgraph}
\end{figure}
\begin{remark}
\label{rem:7ful} Fig. \ref{fig:7ful} together with Lemma \ref{lem:ndf} shows a simple polytope $P\notin\mathcal{F}_1$ with all facets pentagons, hexagons, and only one heptagon, which is adjacent to a pentagon. This polytope can not be obtained by the above seven truncations from a polytope $Q$ in the same class. Indeed, \linebreak $(1;4,5)$-, $(1;5,5)$-, and $(2,6;4,5)$-truncation are not allowed, since $P$ has no quadrangles, and $Q$ can not contain a quadrangle and a heptagon simultaneously.  $Q$ should be obtained from $P$ by a straightening along some edge. This edge can belong either to two pentagons, or to a hexagon and a pentagon. Common edge of two pentagons in $P$ either intersects two pentagons, or a pentagon and a hexagon, or a pentagon and the heptagon. This corresponds to $(2,6;4,4)$-, \linebreak $(2,6;4,5)$-, or $(2,6;4,6)$-truncations; hence $(2,6;5,5)$- and $(2,6;5,6)$-truncations are not allowed. Common edge of a pentagon and a hexagon in $P$ either intersects two pentagons, or  a pentagon and a hexagon, or two hexagons, or a pentagon and the heptagon. This corresponds to $(2,7;4,4)$-, \linebreak$(2,7;4,5)$-, $(2,7;5,5)$-, or $(2,7;5,7)$-truncations, hence \linebreak $(2,7;5,6)$-truncation is also not allowed. In the case of \linebreak $(2,7;5,5)$-truncation $Q$ contains two heptagons; hence this operation is also not allowed.
\end{remark}
\begin{figure}
\includegraphics[width=0.4\textwidth]{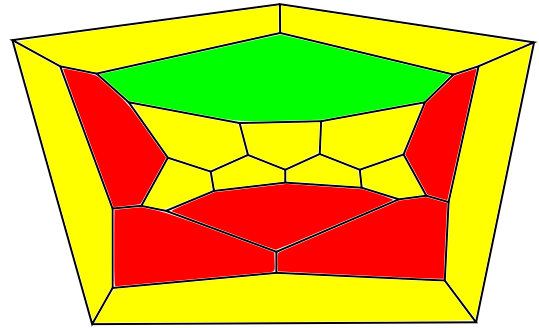}
\caption{Polytope not in $\mathcal{F}_1$ that can not be obtained as a result of our seven truncations}
\label{fig:7ful}
\end{figure}

\subsection{Getting rid of quadrangles}
\label{subsec:grt}
It turns out that in Theorem \ref{th:fultr} we can get rid of quadrangles. 
\paragraph{Two families of fullerenes}
\label{par:fam}
\begin{definition}
\label{def:f1}
\begin{enumerate}
\item Take the dodecahedral cap $C_1$, drawn on Fig. \ref{fig:f1}a). 
\item Add a $5$-belt of hexagons around the patch (Fig. \ref{fig:f1}b). The boundary of the new patch is combinatorially the same.
\item After $k$ steps add the same cap again to obtain a fullerene $D_{5k}$ (The case $k=2$ see on Fig. \ref{fig:f1}c).
\end{enumerate}
Fullerenes $D_{5k}$ for $k\geqslant 1$ are known as  {\it $(5,0)$-nanotubes}.
\begin{figure}
\includegraphics[width=0.4\textwidth]{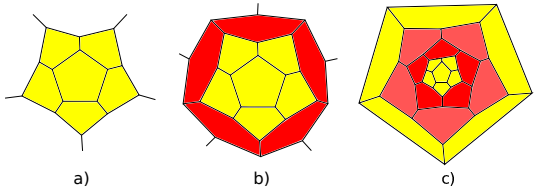}
\caption{Construction of the first family of fullerenes}
\label{fig:f1}       % Give a unique label
\end{figure}
\end{definition}
The insertion of a belt can be realized as a sequence of $(1;4,5)$-, $(1;5,5)$-, and $(2,6;4,5)$-truncations, see Fig. \ref{fig:dkb}
\begin{figure}
\includegraphics[width=0.4\textwidth]{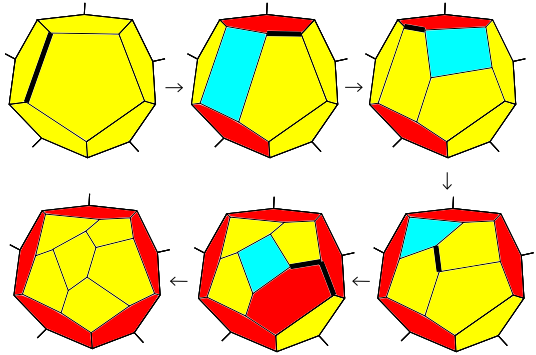}
\caption{Insertion of a belt as a sequence of truncations}
\label{fig:dkb}       % Give a unique label
\end{figure}

\begin{proposition}[\cite{BE15a,BE16}]
\label{prop:f1} 
A fullerene is combinatorially equivalent to $D_{5k}$ for $k\geqslant0$ if and only if it contains the cap $C_1$.
\end{proposition}

The significance of the family $D_{5k}$ is described by the following theorem, which follows directly from \cite{KS08} or  \cite{KM07}.
\begin{theorem}[\cite{BE15b,BE16}]
\label{5belts-theorem} Let $P$ be a fullerene. Then the following statements hold.\\
{\bf I.} Any pentagonal facet is surrounded by a $5$-belt.
There are $12$ belts of this type.\\
{\bf II.} If there is a $5$-belt not surrounding a pentagon, then
\begin{enumerate}
\item it consists only of hexagons;
\item the fullerene is combinatorially equivalent to the polytope $D_{5k}$, $k\geqslant 1$.
\item the number of $5$-belts is $12+k$.
\end{enumerate}
\end{theorem}

\begin{definition}
\label{def:f2}
\begin{enumerate}
\item Take the dodecahedral cap $C_2$ drawn on Fig. \ref{fig:f2} a). 
\item Add $3$ hexagons adjacent to facets with single edge on the boundary (Fig. \ref{fig:f2} b). The boundary of the patch remains combinatorially the same.
\item After $k$ steps ($k=2$ see on Fig. \ref{fig:f2} c) add the same cap again to obtain a fullerene $F_{3k}$ (The case $k=5$ see on Fig. \ref{fig:f2} d).
\end{enumerate}
Fullerenes $F_{3k}$ for $k\geqslant 2$ are also known as {\it $(3,3)$-nanotubes}.
\begin{figure}
\includegraphics[width=0.4\textwidth]{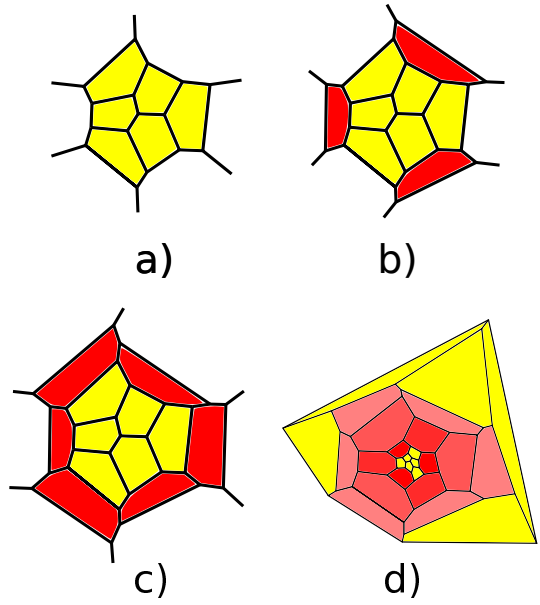}
\caption{Construction of the second family of fullerenes}
\label{fig:f2}       % Give a unique label
\end{figure}
\end{definition}

The addition of three hexagons  can be realized as a sequence of $(1;4,5)$-, $(1;5,5)$-, and $(2,6;4,5)$-truncations, see Fig. \ref{fig:op2}
\begin{figure}
\includegraphics[width=0.4\textwidth]{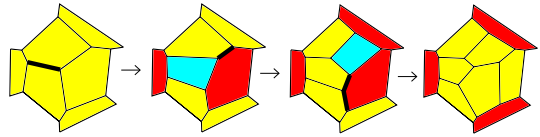}
\caption{Addition of three hexagons as a sequence of truncations}
\label{fig:op2}       % Give a unique label
\end{figure}

\begin{theorem}[\cite{BE15a}]
\label{prop:f2} 
A fullerene is combinatorially equivalent to $F_{3k}$ for $k\geqslant0$ if and only if it contains the cap $C_2$.
\end{theorem}
\begin{proof} First introduce some notions.
\begin{definition}
Let $P$ be a simple polytope. A {\it $k$-loop} is a cyclic sequence $(F_{i_1},\dots,F_{i_k})$, $F_{i_{k+1}}=F_{i_1}$, of facets,  such that $F_{i_1}\cap F_{i_2}$, $\dots$, $F_{i_{k-1}}\cap F_{i_k}$, $F_{i_k}\cap F_{i_1}$ are edges. A $k$-loop is called {\it simple}, if its facets are pairwise different. A simple edge-cycle $\gamma$ {\it borders} a $k$-loop $\mathcal{L}$ if $\mathcal{L}$ is the set of facets that appears when we walk along $\gamma$ in one of the two connected components of $(\partial P)\setminus\gamma$. We say that an $l_1$-loop $\mathcal{L}_1$  {\it borders} an $l_2$-loop $\mathcal{L}_2$ (along $\gamma$), if they border the same edge-cycle $\gamma$. 
\end{definition}

\begin{lemma}
\label{131313-lemma}
Let $P$ be a fullerene, and a $6$-loop\linebreak $\mathcal{L}_1=(F_p,F_t,F_q,F_u,F_v,F_w)$
border a simple $6$-loop $\mathcal{L}_2=(F_i,F_j,F_k,F_l,F_r,F_s)$ with
$(1,3,1,3,1,3)$ edges on the boundary, as drawn on Fig.~\ref{fig:131313} a). Then
$\mathcal{L}_1$ is simple and either forms the fragment on
Fig.~\ref{fig:131313} b), or $F_p,F_q,F_v$ are hexagons and the $6$-loop
$\mathcal{L}_3=(F_j,F_q,F_l,F_v,F_s,F_p)$ is simple and has $(1,3,1,3,1,3)$ edges on
the boundary component intersecting $F_t$.
\end{lemma}
\begin{figure}
\includegraphics[width=0.4\textwidth]{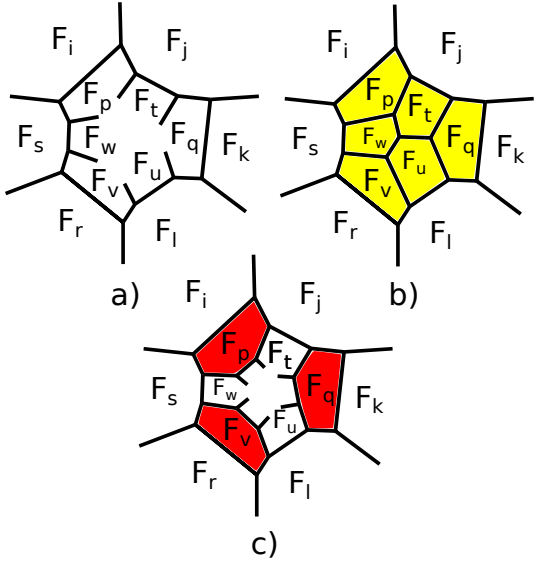}
\caption{Fragments on the fullerene}
\label{fig:131313}
\end{figure}

\begin{proof}
Let $\mathcal{L}_1$ and $\mathcal{L}_2$ border the same simple edge-cycle $\gamma$.

We have $F_p\ne F_q$, since they intersect $F_j$ by different edges. Also
$F_p\cap F_q=\varnothing$, else $F_p\cap F_q\cap F_j$ is a vertex and $F_j$
has only $2$ edges in $\gamma$, but not $3$. Similarly $F_p\cap
F_v=\varnothing$. Then $F_p\ne F_u$, else $F_p\cap F_q\ne\varnothing$. We
have $F_t\ne F_u$, else $F_q$ is a quadrangle. By the~symmetry we obtain that
the $6$-loop $\mathcal{L}_1$ is simple and $F_p\cap F_q=F_q\cap F_v=F_v\cap
F_q=\varnothing$.

Let $F_t\cap F_w\ne\varnothing$. Then $F_p\cap F_t\cap F_w$ is a vertex and
$F_p$ is a pentagon. We have a simple $5$-loop $\mathcal{L}_5=(F_t,F_q,F_u,F_v,F_w)$.
It is not a $5$-belt, else by Theorem \ref{5belts-theorem} it should either
surround a pentagon, and in this case $F_u$ is a quadrangle, or consist of hexagons
and have $(2,2,2,2,2)$ edges on each boundary component. But $F_u$ has
a single edge on the boundary component.  We have $F_t\cap
F_v\ne\varnothing$, else $F_t\cap F_w\cap F_v$ is a vertex and $F_w$ is a
quadrangle. Similarly $F_q\cap F_w=\varnothing$. Also recall that $F_q\cap
F_v=\varnothing$. Then either $F_t\cap F_u\ne\varnothing$, or $F_u\cap
F_w\ne\varnothing$. In the first case $(F_t, F_u,F_v,F_w)$ is a simple
$4$-loop with $F_t\cap F_v=\varnothing$. Since $P$ has no $4$-belts we have $F_u\cap F_w\ne\varnothing$. Then $F_u\cap F_w\cap F_v$, $F_u\cap F_w\cap F_t$, and $F_t\cap F_u\cap F_q$ are vertices and we obtain Fig.~\ref{fig:131313} b). In the
second case we obtain the same picture by the symmetry.

If $F_t\cap F_w=\varnothing$, then $F_p$ is a hexagon.

By the symmetry considering the facets $F_q$ and $F_v$ we obtain that either
all of them are pentagons, and we obtain Fig. \ref{fig:131313}b), or all
of them are hexagons, and we obtain Fig. \ref{fig:131313}c). Since
$F_p\cap F_q=F_q\cap F_v=F_v\cap F_w=\varnothing$, the boundary component of $\mathcal{L}_3$
intersecting $F_t$ is a simple edge-cycle. Since both sets $\{F_j,F_l,F_s\}$
and $\{F_q,F_v,F_p\}$ consist of pairwise different facets and these sets
belong to different connected components with respect to $\gamma$, the
$6$-loop $\mathcal{L}_3$ is simple. It has $(1,3,1,3,1,3)$ edges on the boundary
component intersecting $F_t$.
\qed
\end{proof}
Let $P$ contain the fragment on Fig. \ref{fig:131313}b). Consider the
$6$-loop $\mathcal{L}=(F_p,F_t,F_q,F_u,F_v,F_w)$. The facets $\mathcal{L}\setminus\{F_v\}$ are
pairwise different, since four of them surround the fifth. The same is for
$\mathcal{L}\setminus\{F_p\}$ and $\mathcal{L}\setminus\{F_q\}$. Any two facets of $\mathcal{L}$ belong
to one of these sets, therefore $\mathcal{L}$ is a simple loop. $F_p\cap
F_q=\varnothing$, else $F_p\cap F_q\cap F_t$ is a vertex and $F_t$ is a
quadrangle. Similarly $F_q\cap F_v=F_v\cap F_p=\varnothing$. Then
$\partial \mathcal{L}$ is a simple edge-cycle, and $\mathcal{L}$ has $(1,3,1,3,1,3)$ edges on
it. Applying Lemma \ref{131313-lemma} to $\mathcal{L}$ we obtain the proof of the
theorem.\qed
\end{proof}
\begin{remark} We also have proved that any fragment of the form $C_2$ is bounded by a simple edge-cycle, hence is a patch.
\end{remark}
\paragraph{Construction of fullerenes without quadrangles on the steps}
\label{par:mr}
\begin{theorem}
\label{th:trw4}
Any polytope  in $\mathcal{F}\sqcup \mathcal{F}_1$ can be combinatorially obtained from the dodecahedron by operations $A_1$ -- $A_7$ (see Fig. \ref{fig:7op2}) such that all intermediate polytopes belong to $\mathcal{F}\sqcup \mathcal{F}_1$.  Moreover,
\begin{enumerate}
\item any polytope in $\mathcal{F}\setminus\mathcal{F}^{IPR}$ can be obtained from a polytope in $\mathcal{F}$ by one of  operations $A_1$ -- $A_4$;
\item any polytope in $\mathcal{F}^{IPR}$ can be obtained from a polytope in $\mathcal{F}_1$ by operation $A_6$;
\item any polytope in $\mathcal{F}_1$ can be obtained by operations $A_5$ or $A_7$ from a polytope in $\mathcal{F}$ or $\mathcal{F}_1$ respectively; 
\item operations $A_1$ --  $A_3$ are compositions of $(1;4,5)$-, \linebreak $(1;5,5)$-, and $(2,6;4,5)$-truncations;
\item operations  $A_4$, $A_5$, $A_6$, $A_7$ are $(2,6;5,5)$-, $(2,6;5,6)$-, $(2,7;5,5)$-, and $(2,7;5,6)$-truncations respectively.

\end{enumerate}
\end{theorem}
\begin{figure}
\includegraphics[width=0.4\textwidth]{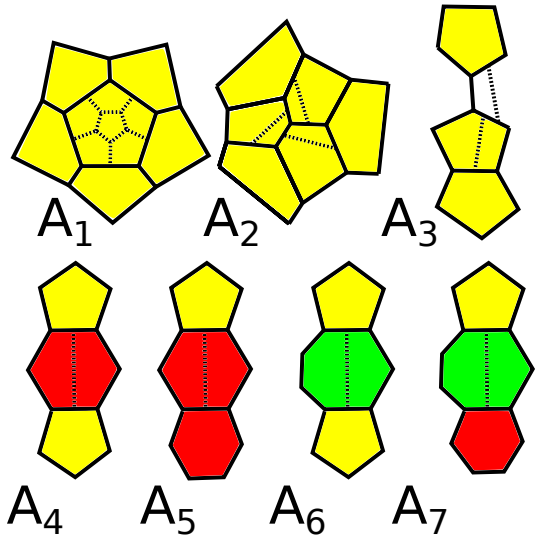}
\caption{Seven operations sufficient to construct any polytope in $\mathcal{F}\sqcup \mathcal{F}_1$ from the dodecahedron. Dotted lines correspond to edges on the resulting polytope}
\label{fig:7op2}       % Give a unique label
\end{figure}
\begin{proof} We will need the following result.
\begin{lemma}\label{lem:55} If a fullerene $P$ has two adjacent pentagons, then either it contains patch $C_1$, or patch $C_2$, or one of the patches drawn on Fig. \ref{fig:65556}.
\end{lemma}
\begin{figure}
\includegraphics[width=0.4\textwidth]{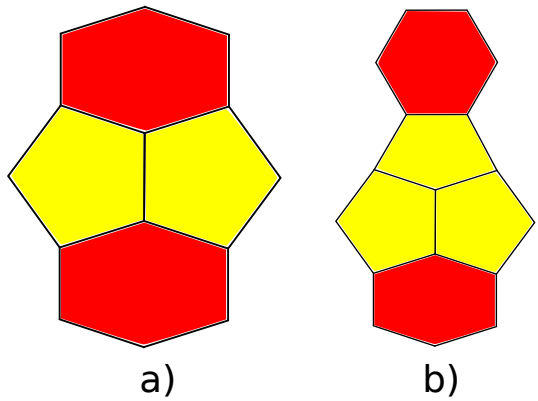}
\caption{a) Patch $P_1$; b) Patch $P_2$}
\label{fig:65556}       % Give a unique label
\end{figure}
\begin{proof}
Since any fullerene has no $3$- and $4$-belts, any fragment on a fullerene of the form $C_1$, $C_2$, $P_1$, or $P_2$ is bounded by a simple edge-cycle; hence it is a patch. 

By assumption $P$ contains two adjacent pentagons $F_i$ and $F_j$. 
Consider the facets $F_k$, $F_l$ such that $F_i\cap
F_j\cap F_k$ and $F_i\cap F_j\cap F_l$ are vertices. If both facets $F_k$,
$F_l$ are hexagons, then $P$ contains the fragment $P_1$. Else one of them is a
pentagon. Without loss of generality let it be $F_k$. We obtain three
pentagons with a common vertex, as drawn on Fig.~\ref{fig:NIP} a).

If $F_p$, $F_q$, $F_l$ are pentagons, we obtain the fragment $C_2$.

Let two of them, say $F_p$ and $F_q$, be hexagons. Then either $F_u$ is a
pentagon and we obtain the fragment $P_1$, or $F_u$ is a hexagon, and we obtain  Fig. \ref{fig:NIP} b). 
If $F_v$, $F_l$, or $F_w$ is a hexagon, we obtain the
fragment $P_2$. If all of them are pentagons, we obtain the fragment $C_2$.

Now let one of the facets $F_p$, $F_q$, $F_l$, say $F_l$, be a hexagon, and
two others -- pentagons, as drawn on Fig. \ref{fig:NIP} c). Then either $F_u$ is a
hexagon and we obtain the fragment $P_2$, or it is a pentagon and we obtain the
fragment $C_1$.
\begin{figure}
\includegraphics[width=0.4\textwidth]{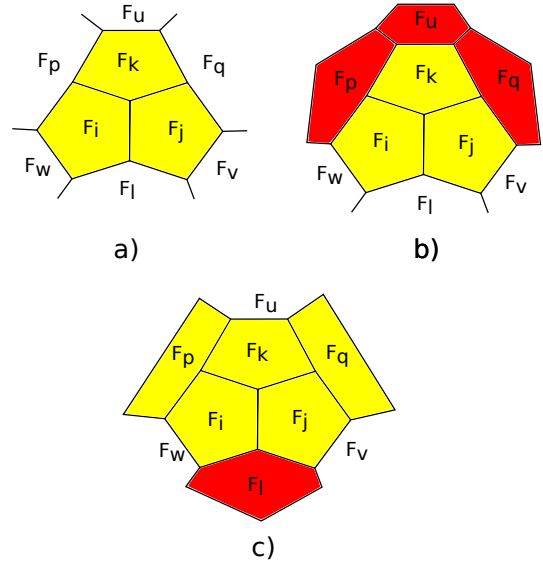}
\caption{Fragments on a fullerene with adjacent pentagons}
\label{fig:NIP}       
\end{figure}\qed
\end{proof}
If a fullerene $P$ has two adjacent pentagons, then by Lemma \ref{lem:55} it contains one of the patches $C_1$, $C_2$, $P_1$ or $P_2$. If $P$ contains patch $C_1$ or $C_2$, then by Proposition \ref{prop:f1} or Theorem \ref{prop:f2} it is contained in the family $D_{5k}$ or $F_{3k}$ respectively, hence can be obtained from the dodecahedron by operations $A_1$ or $A_2$. Fig. \ref{fig:dkb} and Fig. \ref{fig:op2} show how operations $A_1$ and $A_2$ are decomposed into $(1;4,5)$-, $(1;5,5)$-, and $(2,6;4,5)$-truncations. 

If $P$ contains patch $P_1$, then by Corollary \ref{cor:ful} $P$ is obtained from some other fullerene $Q$ by operation $A_4$, which is a $(2,6;5,5)$-truncation.

Fig. \ref{fig:a4d} shows how operation $A_3$ decomposes into $(1;5,5)$- and $(2,6;4,5)$-truncations. If $P$ contains patch $P_2$, then applying the corresponding inverse operations by Corollary  \ref{cor:ful} we see that the fullerene is obtained from some other fullerene by operation $A_3$.

If a fullerene $P$ has no adjacent pentagons, then the proof of Theorem 9.12 in \cite{BE16} implies that $P$ is obtained from a polytope in $\mathcal{F}_1$ by operation $A_6$, which is a $(2;7;5,5)$-truncation.

Theorem \ref{th:fultr} implies that any polytope in $\mathcal{F}_1$ can be obtained by operations  $A_5$ and $A_7$, from a fullerene or a polytope in $\mathcal{F}_1$ respectively. Here $A_5$ is a $(2,6;5,6)$-truncation, and $A_7$ is a $(2,7;5,6)$-truncation.
\begin{figure}
\includegraphics[width=0.4\textwidth]{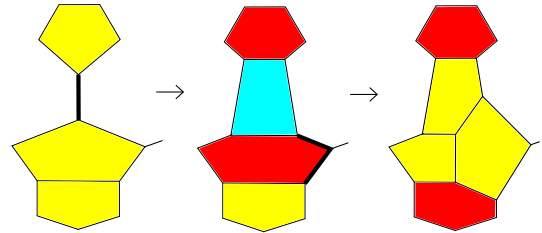}
\caption{Decomposition of operation $A_3$ into truncations}
\label{fig:a4d}       
\end{figure}\qed
\end{proof}
\subsection{Singular $IPR$-fullerenes}
\label{subsec:sipr}
\begin{theorem}
\label{th:iptr}
Any polytope in $\mathcal{F}\sqcup \mathcal{F}_1^{IPR}$ can be combinatorially obtained from the dodecahedron by operations $A_1$ -- $ A_4$, $A_6$, $A_7$, $B_1$ -- $B_5$ (Fig. \ref{fig:11op}) in such a way that all intermediate polytopes belong to $\mathcal{F}\sqcup \mathcal{F}_1^{IPR}$. Moreover,
\begin{enumerate}
\item any polytope in $\mathcal{F}\setminus\mathcal{F}^{IPR}$ can be obtained from a polytope in $\mathcal{F}$ by one of  operations $A_1$ -- $A_4$;
\item any polytope in $\mathcal{F}^{IPR}$ can be obtained by operations $B_1$ or  $B_3$ from a polytope in $\mathcal{F}$, or by operation $A_6$   from a polytope in $\mathcal{F}_1^{IPR}$;
\item any polytope in $\mathcal{F}_1^{IPR}$ can be obtained by operations $B_2$, $B_4$, or $B_5$ from a polytope in $\mathcal{F}$, or by operation $A_7$ from a polytope in $\mathcal{F}_1^{IPR}$;  
\item operations $A_1$ --  $A_3$ are compositions of  $(1;4,5)$-,\linebreak $(1;5,5)$-, and  $(2,6;4,5)$-truncations;
\item operations  $A_4$, $A_6$, $A_7$ are $(2,6;5,5)$-, $(2,7;5,5)$-, and $(2,7;5,6)$-truncations respectively;
\item operations $B_1$ -- $B_5$ are compositions of $(2,6;5,6)$-,\linebreak $(2,7;5,5)$-, and $(2,7;5,6)$-truncations.
\end{enumerate}
\end{theorem}
\begin{figure}
\includegraphics[width=0.4\textwidth]{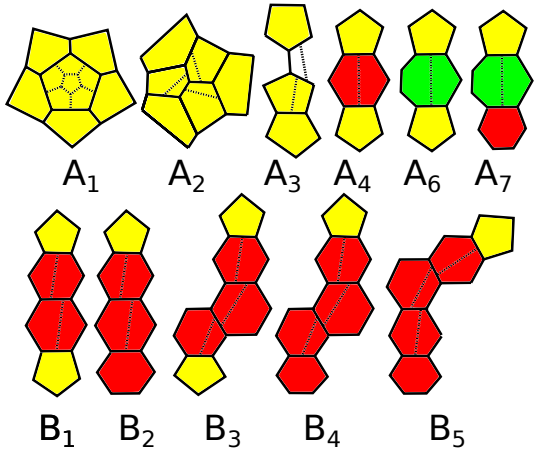}
\caption{Eleven operations sufficient to construct any polytope in $\mathcal{F}\sqcup \mathcal{F}_1^{IPR}$ from the dodecahedron. Dotted lines correspond to edges of the resulting polytope}
\label{fig:11op}       % Give a unique label
\end{figure}
\begin{proof}
The statement 1. is a part of Theorem \ref{th:trw4}.

To prove the statement 2. take an $IPR$-fullene $P$. Consider $3$ cases.

{\it Case 1.} $P$ has an edge intersecting by vertices two pentagons. Since $P$ has no $3$- and $4$-belts, it contains the patch on Fig. \ref{fig:5e5} a). Then Fig. \ref{fig:5e5} b)--c) show how this patch can be reduced by straightening along edges. Then composition of the inverse operations gives operation $B_1$.
\begin{figure}
\includegraphics[width=0.4\textwidth]{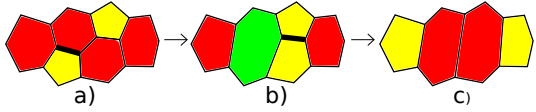}
\caption{a) A patch on the fullerene; b)--c) it's reductions}
\label{fig:5e5}       
\end{figure}

{\it Case 2.} $P$ has no edge intersecting two pentagons, but has a hexagon adjacent to two pentagons by opposite edges. The absence of $3$- and $4$-belts, and the structure of $5$-belts given by Theorem \ref{5belts-theorem} implies that $P$ contains the patch on Fig. \ref{fig:5s5} a). Then Fig. \ref{fig:5s5} b)-d) show how this patch can be reduced by straightening along edges. Then composition of the inverse operations gives operation $B_3$.
\begin{figure}
\includegraphics[width=0.4\textwidth]{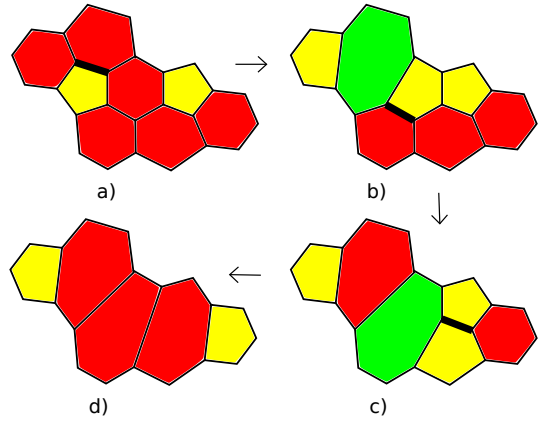}
\caption{a) A patch on the fullerene; b)-d) it's reduction}
\label{fig:5s5}       
\end{figure}

{\it Case 3.} $P$ has no hexagons adjacent to two pentagons.  Then any pentagon is surrounded by a belt of hexagons  intersecting no other pentagons. Then any straightening along common edge of a pentagon and a hexagon gives a polytope in  $\mathcal{F}_1^{IPR}$ such that $P$ is obtained from it by operation $A_6$, which is a $(2,7;5,5)$-truncation. 

Now consider a polytope $P\in\mathcal{F}_1^{IPR}$ and a pentagon adjacent to the heptagon, as drawn on Fig. \ref{fig:75} a). Straightening the edge we obtain Fig. \ref{fig:75} b). Then $P$ is obtained from  the resulting polytope $Q$ by  operation $A_7$.  The polytope $Q$ does not belong to $\mathcal{F}^{IRP}_1$ if and only if one of the facets $F_u$, $F_v$, or $F_w$ is a pentagon.  
\begin{figure}
\includegraphics[width=0.4\textwidth]{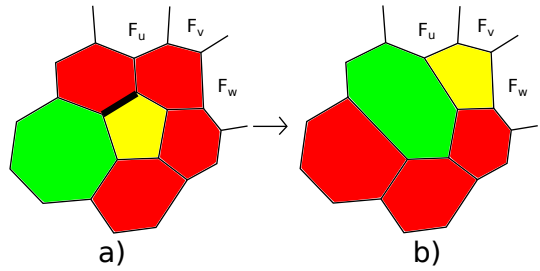}
\caption{a) A patch on the polytope in $\mathcal{F}_1^{IPR}$,  b) it's reduction}
\label{fig:75}       
\end{figure}
If $F_u$ us a pentagon, then since $P$ has no $3$- and $4$-belts, we obtain the patch drawn on Fig. \ref{fig:751} a). Applying the straightenings as drawn on Fig. \ref{fig:751} b)--c), we obtain a fullerene $Q$ such that $P$ is obtained from $Q$ by operation $B_2$. 

\begin{figure}
\includegraphics[width=0.4\textwidth]{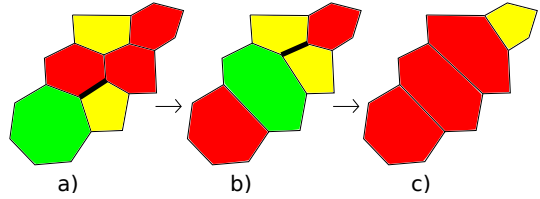}
\caption{a) A patch on the polytope in $\mathcal{F}_1^{IPR}$,  b)-c) it's reductions}
\label{fig:751}       
\end{figure}

If $F_v$ is a pentagon, then since $P$ has no $3$- and $4$-belts, we obtain the fragment drawn on Fig. \ref{fig:752} a). It is a patch if and only if $F_i\cap F_j=\varnothing$. Applying the straightenings as drawn on Fig. \ref{fig:752} b)--d), we obtain a fullerene $Q$.  Theorem \ref{5belts-theorem} implies that $F_i\cap F_j=\varnothing$; hence the same if true for  $P$. Thus, we have a patch, and the polytope $P$ is obtained from $Q$ by operation $B_4$. 
\begin{figure}
\includegraphics[width=0.4\textwidth]{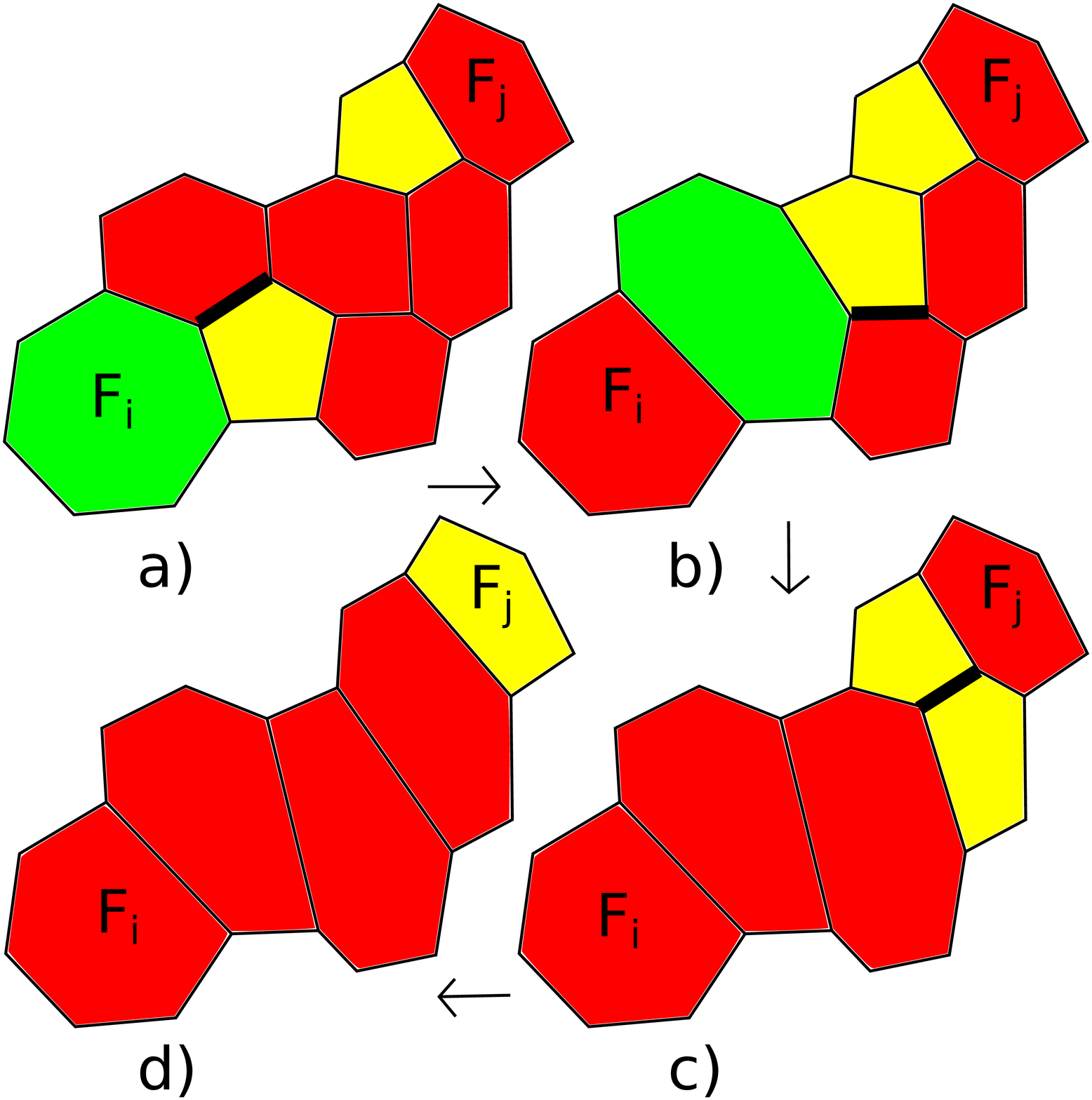}
\caption{a) A patch on the polytope in $\mathcal{F}_1^{IPR}$,  b)-d) it's reductions}
\label{fig:752}       
\end{figure}

If $F_v$ is a hexagon and $F_w$ is a pentagon, then since $P$ has no $3$- and $4$-belts, we obtain the fragment drawn on Fig. \ref{fig:753} a). It is a patch if and only if $F_i\cap F_j=\varnothing$. Applying the straightenings as drawn on Fig. \ref{fig:753} b)--d), we obtain a fullerene $Q$.  Theorem \ref{5belts-theorem} implies that $F_i\cap F_j=\varnothing$ in $Q$; hence the same if true for  $P$. Thus, we have a patch, and  the polytope $P$ is obtained from $Q$ by operation $B_5$.  This finishes the proof.
\begin{figure}
\includegraphics[width=0.4\textwidth]{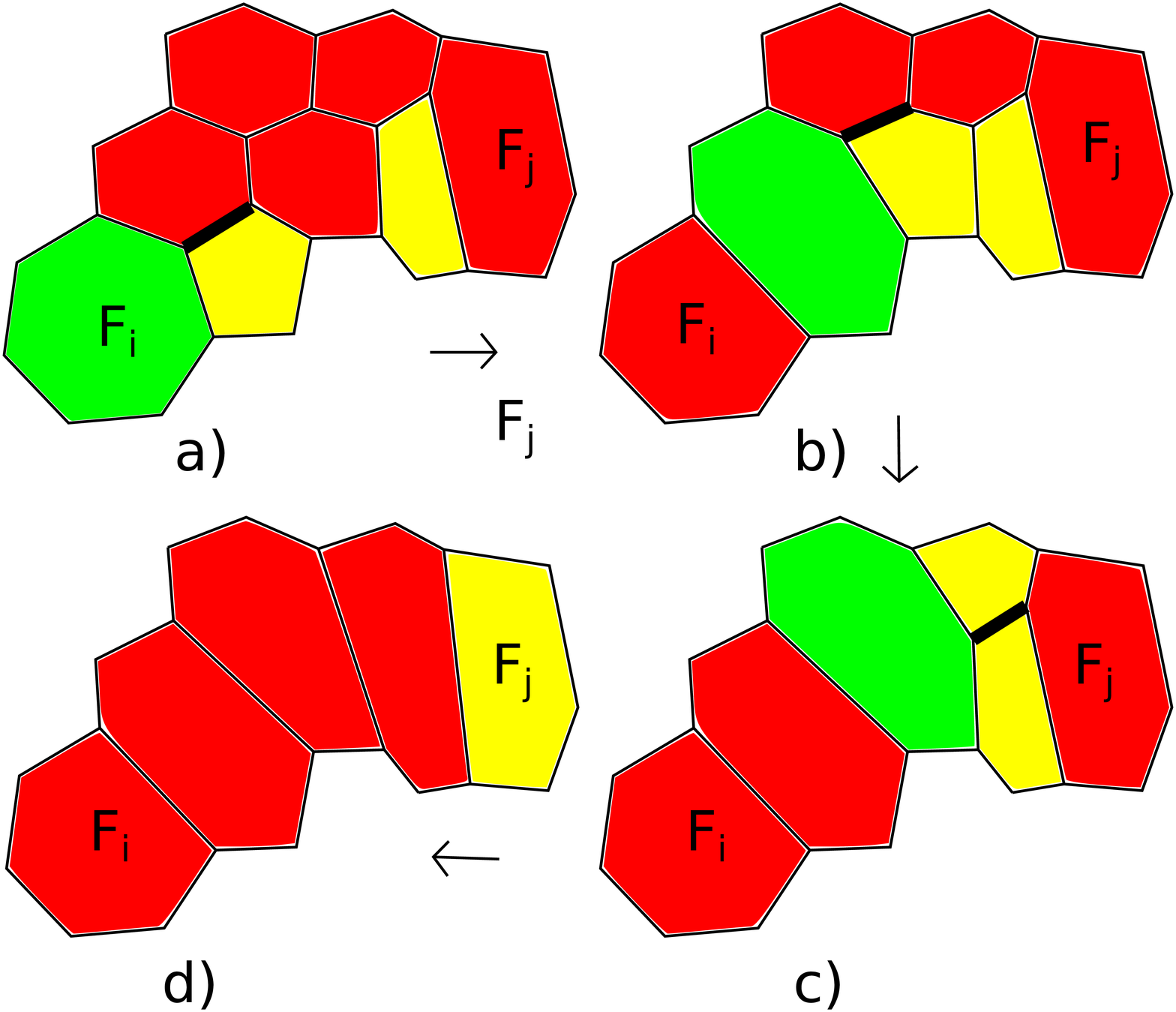}
\caption{a) A patch on the polytope in $\mathcal{F}_1^{IPR}$,  b)-d) it's reductions}
\label{fig:753}       
\end{figure}\qed
\end{proof}

%\label{sec:1}
%Text with citations \cite{RefB} and \cite{RefJ}.
%\subsection{Subsection title}
%\label{sec:2}
%as required. Don't forget to give each section
%and subsection a unique label (see Sect.~\ref{sec:1}).
%\paragraph{Paragraph headings} Use paragraph headings as needed.
%\begin{equation}
%\end{equation}

%\begin{figure}
%  \includegraphics{example.eps}
%\end{figure}
%\begin{figure*}
%  \includegraphics[width=0.75\textwidth]{example.eps}
%\end{figure*}
%\begin{table}
%\begin{tabular}{lll}
%\hline\noalign{\smallskip}
%\end{tabular}
%\end{table}

\section{Compliance with Ethical Standards}
\begin{acknowledgements}
This work is supported by the Russian Science Foundation under grant 14-11-00414. The second author is the Young Russian Mathematics award winner. 
\end{acknowledgements}
{\small \noindent{\bf Conflict of Interest:} The authors declare that they have no conflict of interest.}



\end{document}